\documentclass{article}
\usepackage{amsfonts}
\usepackage{amsmath}

\newtheorem{theorem}{Theorem}

\newtheorem{corollary}[theorem]{Corollary}

\newtheorem{definition}[theorem]{Definition} 
\newtheorem{example}[theorem]{Example}

\newtheorem{proposition}[theorem]{Proposition}
\newtheorem{remark}[theorem]{Remark}

\newenvironment{proof}[1][Proof]{\noindent\textbf{#1.} }{\ \rule{0.5em}{0.5em}}

\begin{document}
	\title{On ruled normal surfaces associated with rectifying curves}
	\author{Ana-Maria Boldeanu}
	\date{}
	\maketitle
	
	\begin{abstract}
		In this paper we study the geometrical entities as Gauss and mean curvatures attached to the ruled normal surfaces associated with rectifying curves. A particular and special example of rectifying curves is also investigated, via its attached ruled normal surface. Our main result is the proof of the non-existence of rectifying curves that are simultaneously Bertrand curves.
	\end{abstract}
	
	
	
	\section{Introduction}
	
	In differential geometry, a rectifying curve is a spatial curve characterized by the property that all its rectifying planes intersect at a single fixed point. When this point coincides with the origin of a Cartesian coordinate system, the position vector of the curve lies entirely within its rectifying plane. A characteristic of these curves is that the ratio of their torsion to their curvature varies as a nonconstant linear function of the arc length parameter.
	
	The term \textit{rectifying curve} was introduced by Chen in \cite{Chen} to designate a curve whose position vector always lies in its rectifying plane. Kim et al. \cite{Kim-Chu-Cho} studied the properties of the “space curves satisfying \( \tau/\kappa=as+b \)”. They proved that a curve has this property if and only if there exists a point such that all of the rectifying planes to the curve pass through this point. By a translation, any curve of this type turns out to be a rectifying curve in the restricted sense of \cite{Chen}. In this paper, we consider, for a rectifying curve, the definition from \cite{Kim-Chu-Cho} because it employs only the intrinsic properties of the curve. Generally, the geometric properties of rectifying curves are detailed presented in \cite{Chen}, \cite{Jianu}.
	
	The main aim oh this paper is to investigate the possible connections between rectifying and Bertrand curves in order to obtain the more precisely descriptions of these curves. Consequently, we obtained that there is no rectifying curve which is at the same time a Bertrand curve. In this direction, the outline of the paper is the following: in Section 2 we present the basic notions and formulas that characterize the Bertrand and rectifying curves; in Section 3 we investigate in detail the geometrical entities associated with a particular curve which is very suggestive for rectifying curves. We also prove that the sets of rectifying curves and Bertrand curves are disjoint ones, via the minimality of the ruled normal surfaces associated with a spatial curve (see \cite{Izu-Take}).

	
	\section{Preliminaries}
	
	We now review some basic concepts on classical geometry of space curves and ruled surfaces in Euclidean space. We denote by $\left< \; ,\right>$ the inner product of the Euclidean space. Let $I$ be a real interval and $\gamma : I \rightarrow \mathbb{R}^3$ be a space curve of class $C^{\infty}$ with $\gamma'(t) \neq 0$, where $\gamma'(t)=\dfrac{d\gamma}{dt}(t)$. We also denote the norm of $\gamma$ by $\left\| \gamma \right\|$. The \textit{arc-lenght} of a curve $\gamma$, measured from $\gamma(t_0), \; t_0 \in I$ is 
	\[
	s(t)=\int_{t_0}^{t} \left\|\gamma'(\sigma)\right\| d\sigma.
	\]
	Then a parameter $s$, $s\in [0,L]$, where $L$ is the \textit{length} of the curve, is determined such that $\left\|\gamma'(s) \right\|=1$. Generally, we say that a curve $\gamma$ is \textit{parametrized by the arc-lenght} if it satisfies $\left\|\gamma'(s) \right\|=1$. The \textit{unit} vector field $T(s)=\gamma'(s)$ is called the \textit{tangent vector} of $\gamma$ at $s$. We define the  \textit{curvature} of the curve $\gamma$ by $\kappa(s)=\left\| \gamma ''(s)\right\|$. If $\kappa(s) \neq 0$, then the \textit{unit principal normal vector} $N(s)$ of the curve $\gamma$ at $s$ is given by $\gamma ''(s)=\kappa(s)N(s)$. Since $\left< T(s), T(s)\right> =1$, $\forall \; s \in [0,L]$, it follows that $\left< T'(s), T(s)\right> =0$; hence, $N(s)$ is orthogonal to $T(s)$. The \textit{unit} vector $B(s)=T(s) \times N(s)$ is called the \textit{binormal vector} of the curve $\gamma$ at $s$, which is a vector field orthogonal to both $T(s)$ and $N(s)$. The orthogonal unit vectors $T(s), N(s)$ and $B(s)$ form the well-known \textit{Fr\'{e}net frame} or \textit{Fr\'{e}net trihedron}.
	
	It is a well-known fact that the following \textit{Fr\'{e}net-Serret} formula holds:
	\begin{equation} \label{F-S}
		\begin{matrix} 
			T'(s) =	& \kappa(s) N(s)  \medskip\\
			N'(s) =	& -\kappa(s)T(s)+\tau(s)B(s)  \medskip\\
			B'(s) =	& -\tau(s)N(s),  \\
		\end{matrix}
	\end{equation}
	where the \textit{torsion} of the curve $\gamma$ at $s$ is 
	$$\tau(s)=\dfrac{det(\gamma'(s),\gamma''(s),\gamma'''(s))}{ \left\|\gamma'(s) \times \gamma''(s)  \right\|^2}.$$
	We recall that, at each point of the curve, the planes determined by $\left\{T(s),N(s) \right\}$, $\left\{T(s),B(s) \right\}$ and $\left\{N(s),B(s) \right\}$ are known as the \textit{osculating plane}, the \textit{rectifying plane} and the \textit{normal plane}, respectively. 
	
	\begin{definition}
		\emph{(\textbf{\cite{Naka-Taka}})}. The curves $\gamma$ and $\bar{\gamma}$ are \textbf{ Bertrand mates} if there exists a smooth function $\lambda :I \rightarrow \mathbb{R}$ such that $\bar{\gamma}(t)=\gamma(t)+\lambda(t)N(t)$ and $N(t)=\pm\bar{N}(t)$, for all $t \in I$. Moreover, we say that $\gamma : I \rightarrow \mathbb{R}^3 $ is a \textbf{Bertrand curve} if there exists another non-degenerate curve $\bar{\gamma} : I \rightarrow \mathbb{R}^3 $ such that $\gamma$ and $\bar{\gamma}$ 
		are Bertrand mates.
		
	\end{definition}
	
	\begin{theorem}
		\emph{(\textbf{\cite[p.~96]{Lip}})}. A curve $\gamma : I \rightarrow \mathbb{R}^3$ is a \textbf{Bertrand curve} if and only if there exist nonzero real numbers $A, B$ such that
		\[
		A\kappa(s)+B\tau(s)=1, \; \forall \; s \in [0,L].
		\]
	\end{theorem}
	
	\begin{definition}
		\emph{(\textbf{\cite{Izu-Take}})}.
		The vector $D(s)=\tau(s)T(s)+\kappa(s)B(s)$, associated with an unit speed curve $\gamma:[0,L] \rightarrow \mathbb{R}^3$,  is called the \textbf{Darboux vector field} of $\gamma$.
	\end{definition} 
	Note that, using the Darboux vector field, the Fr\'{e}net-Serret formula can be rewritten as follows:
	\[
	\begin{matrix}
		T'(s) =  D(s) \times T(s)  \medskip\\
		N'(s) =  D(s) \times N(s) \medskip\\
		B'(s) =  D(s) \times B(s). \\
	\end{matrix}
	\]
	\begin{definition}
		\emph{(\textbf{\cite{Izu-Take}})}. The vector field $$
		\tilde{D}(s)=\dfrac{\tau}{\kappa}(s)T(s)+B(s)
		$$
		along $\gamma$, under the condition that $\kappa(s) \neq 0$, $ \forall \; s \in [0,L]$, is called the \textbf{modified Darboux vector field} of $\gamma$.
	\end{definition}

	\begin{definition}
		\emph{(\textbf{\cite{Izu-Take}})}. A map $F_{(\gamma, \delta)}:I \times \mathbb{R} \rightarrow \mathbb{R}^3$, defined by 
		$$F_{(\gamma, \delta)}(t,u)=\gamma(t)+u\delta(t),$$
		where $\gamma : I \rightarrow \mathbb{R}^3$, $\delta : I \rightarrow \mathbb{R}^*$ are smooth mappings and $I$ is an open interval, is called a \textbf{ruled surface} in $\mathbb{R}^3$. The curve $\gamma$ is called the \textbf{base curve} and $\delta$ is called the \textbf{director curve}. Note that the straightlines $u \mapsto \gamma(t)+u\delta(t)$ are called \textbf{rulings}.
		We also remind that a ruled surface associated to a spatial curve $\gamma$ with $\kappa(s)\neq0$, $\forall \; s \in [0,L]$, which is defined by 
		$$F_{(\gamma,\tilde{D})}(s,u)=\gamma(s)+u\tilde{D}(s),$$
		is called the \textbf{rectifying developable} of $\gamma$. Particularly, the ruled surface 
		\begin{equation} \label{Prin-Norm-Surf}
			F_{(\gamma,N)}(s,u)=\gamma(s)+uN(s),
		\end{equation}
		is called the \textbf{principal normal surface} of $\gamma$.
	\end{definition}
	
	\begin{definition} \label{def R-C}
		\emph{(\textbf{\cite{Chen}})}.  A unit curve $\gamma: [0,L] \rightarrow \mathbb{R}^3$, defined by
		\begin{equation} \label{Rec-Curve}
			\gamma(s)=\lambda(s)T(s)+\mu(s)B(s),
		\end{equation}
		where $\lambda(s), \; \mu(s)$ are some functions, is called a \textbf{rectifying curve} in $\mathbb{R}^3$.
	\end{definition}
	
	The following theorem provides the connection between the torsion and curvature that characterizes a rectifying curve.

	\begin{theorem} \label{Th R-C}
		\emph{(\textbf{\cite{Kim-Chu-Cho}})}. The unit speed curve $\gamma : [0,L] \rightarrow \mathbb{R}^3$ is a rectifying curve if and only if there exist the constants $a_1, a_2 \in \mathbb{R}, \; a_1 \neq 0$, such that
		\begin{equation} \label{R-C cond}
			\dfrac{\tau(s)}{\kappa(s)} = a_1s+a_2,\;\forall \; s\in [0,L].
		\end{equation}  
	\end{theorem}
	
	\begin{remark}
		By Definition $\ref{def R-C}$, Theorem \ref{Th R-C} and Fr\'{e}net-Serret formulas (\ref{F-S}), we remark that a rectifying curve can be written as:
		\begin{equation} \label{New R-C}
			\gamma(s)=\left ( s + \dfrac{a_2}{a_1} \right )T(s)+\dfrac{1}{a_1}B(s),
		\end{equation}
		for any $s \in [0,L]$ and $a_1, a_2 \in \mathbb{R}, \; a_1 \neq 0$. 
	\end{remark}
	
	\section{Ruled surfaces associated with rectifying curves}
	
	In this section we discuss about the normal ruled surface associated with a rectifying curve, and we compute their Gaussian curvature and mean curvature, representing new results of this paper.
	
	Firstly, we consider the \textit{ruled normal surface associated with the rectifying curve} $\gamma(s)$ from (\ref{New R-C}). This ruled normal surface is parametrized as follows:
	\begin{equation} \label{R-N-S}
		F_{(\gamma,N)}(s,u)=\left ( s + \dfrac{a_2}{a_1} \right )T(s)+\dfrac{1}{a_1}B(s)+uN(s),
	\end{equation}
	for any $s \in [0,L]$, $u \in \mathbb{R} $ and $a_1, a_2 \in \mathbb{R}, \; a_1 \neq 0$.
	
	In the next proposition we calculate the Gaussian and mean curvatures of the ruled normal surface (\ref{R-N-S}):
	\begin{proposition}
		The curvatures of the ruled normal surface associated with the rectifying curve (\ref{New R-C}), which is parametrized by (\ref{R-N-S}), are:
		\begin{itemize}
			\item [1.] The \textbf{Gaussian curvature} $K(s,u)$ is given by
			\begin{equation} \label{G-C}
				K(s,u)=-\dfrac{\tau^2(s)}{\left [ (1-u\kappa(s))^2 + (u\tau(s))^2 \right ]^2}, \quad \forall \; s \in [0,L], u \in \mathbb{R};
			\end{equation}
			
			\item [2.] The \textbf{mean curvature} $H(s,u)$ is given by
			\begin{equation} \label{M-C}
				H(s,u)=\dfrac{u\left [ \tau'(s)+u(\kappa'(s)\tau(s)-\kappa(s)\tau'(s)) \right ]}{2\left [ (1-u\kappa(s))^2+(u\tau(s))^2 \right ]^{3/2}}, \quad \forall \; s \in [0,L], u \in \mathbb{R}. 
			\end{equation}
		\end{itemize}
	\end{proposition} 
	
	\begin{proof}
		We remind the formulas for Gaussian curvature and mean curvature:
		\begin{equation*}
			K=\frac{ln-m^2}{EG-F^2}, \quad
			H=\dfrac{1}{2}\cdot\dfrac{lG-2mF+nE}{EG-F^2},
		\end{equation*}
		where $E, \; F, \; G$ are the coefficients of the first fundamental form and $l, \; m, \; n$ are the coefficients of the second fundamental form. 
		
		By direct computation, we obtain
		\[
		E=(1-u\kappa(s))^2+(u\tau(s))^2, \; F=0, \; G=1,
		\]
		\[
		l=\dfrac{u[\tau'(s)+u(\kappa'(s)\tau(s)-\kappa(s)\tau'(s))]}{\sqrt{E}}, \; m=\dfrac{\tau(s)}{\sqrt{E}}, \; n=0,
		\]
		where $ s \in [0,L], u \in \mathbb{R}$.
		Replacing into the curvature formulas mentioned above, we obtain what we were looking for. 
	\end{proof}
	\begin{corollary}
		All points of the ruled normal surface associated with a rectifying curve are hyperbolic ones.
	\end{corollary}
	\begin{proof}
		Because we have for any rectifying curve $\tau(s)\neq 0, \;\forall \; s \in [0,L] $, from formula (\ref{G-C}) it follows that the inequality $K(s,u)<0$, $\forall \; u\in \mathbb{R}$, $s \in [0,L]$, holds good.
	\end{proof}
	
	\begin{proposition} \label{Not minimal}
		The ruled normal surface associated with a rectifying curve is never minimal one.
	\end{proposition}
	
	\begin{proof}
		We suppose that the ruled normal surface associated with a rectifying curve is minimal, that is $H\equiv0$.
		From (\ref{M-C}), it follows that we obtain for all $s \in [0,L]$:
		$$\left\{ \begin{array}{l}
			\tau'(s)=0 \\
			\kappa'(s)\tau(s)-\kappa(s)\tau'(s)=0
		\end{array} \right. \Leftrightarrow
		\left\{ \begin{array}{l}
			\tau(s)=C_1, \; C_1\in \mathbb{R} \medskip\\
			\dfrac{\kappa'(s)}{\kappa(s)}=\dfrac{\tau'(s)}{\tau(s)}
		\end{array} \right. \Leftrightarrow	$$
		$$
		\Leftrightarrow \left\{ \begin{array}{l}
			\tau(s)=C_1, \; C_1 \in \mathbb{R} \medskip\\
			\kappa(s)=C_2\tau(s), \; C_2 >0.
		\end{array} \right.
		$$
		
		In conclusion, we infer that we have 
		$$\dfrac{\tau(s)}{\kappa(s)}= \mbox{constant},$$ which contradicts the hypothesis that the curve is rectifying.
	\end{proof}
	
	In the sequel, we provide a proper example of a rectifying curve, which relies on the finding of a particular solution for the differential equations system produced by the relation (\ref{Rec-Curve}).  
	\begin{example}
		The curve $\gamma : [0,L] \rightarrow \mathbb{R}^3, \; \gamma(s)=(x(s),y(s),z(s))$, defined by
		\begin{equation} \label{Ex 1}
			\begin{array}{l}
				x(s) =   \pm \dfrac{b_1}{a_1} \ln \left |(a_1s+a_2)+\sqrt{1+(a_1s+a_2)^2} \right | + C_3,  \medskip\\
				y(s) =  \pm \dfrac{b_2}{a_1} \ln \left |(a_1s+a_2)+\sqrt{1+(a_1s+a_2)^2} \right | + C_4,   \medskip\\
				z(s) =   \pm \dfrac{\sqrt{1+(a_1s+a_2)^2}}{a_1} + C_5, \\
			\end{array}
		\end{equation} 
		where $b_1, b_2 \in [-1,1], \; b_1^2+b_2^2=1$, $a_1, a_2, C_3, C_4, C_5 \in \mathbb{R}$ and $a_1 \neq 0$, has the following torsion and curvature:
		\begin{equation} \label{C-T}
			\tau(s) = \pm \dfrac{2a_1(a_1s+a_2)}{1+(a_1s+a_2)^2}, \quad \kappa(s) = \pm \dfrac{a_1}{1+(a_1s+a_2)^2}.
		\end{equation}
		In other words, we have $\tau(s)/\kappa(s)=2(a_1s+a_2)$, that is $\gamma(s)$ is a rectifying curve.
	\end{example}	
	
	In what follows, we prove the relations (\ref{C-T}).
	
	Taking to account that $T(s)=\gamma'(s)$ and $B(s)=\gamma'(s) \times \gamma''(s)/\kappa(s)$, we deduce that the rectifying curve (\ref{New R-C}) verifies the following differential equations system: 
	\[
	\left\{\begin{matrix}
		x(s) = \left ( s+\dfrac{a_2}{a_1} \right )x'(s)+\dfrac{y'(s)z''(s)-y''(s)z'(s)}{\kappa(s)a_1}  \\ \\ 
		y(s)  = \left ( s+\dfrac{a_2}{a_1} \right )y'(s)-\dfrac{x'(s)z''(s)-x''(s)z'(s)}{\kappa(s)a_1} \\ \\
		z(s)  = \left ( s+\dfrac{a_2}{a_1} \right )z'(s)+\dfrac{x'(s)y''(s)-x''(s)y'(s)}{\kappa(s)a_1},\\
	\end{matrix}\right.
	\]
	where $a_1, a_2 \in \mathbb{R}, \; a_1 \neq 0$
	and 
	\[
	(x'(s))^2 + (y'(s))^2 + (z'(s))^2  =   1,\quad \forall \; s \in [0,L].
	\]
	We are looking for a solution of the above system of the form
	\[
	\left\{\begin{array}{l}
		x'(s)= \cos \theta \cos \varphi  \\
		y'(s)=\cos \theta \sin \varphi \\
		z'(s)= \sin \theta,  \\
	\end{array}\right.
	\]
	where $\theta \in \left [ -\dfrac{\pi}{2}, \dfrac{\pi}{2} \right ]$, $\varphi \in [0, 2\pi]$.  
	By direct and laborious calculations, we get 
	\[
	z(s)=\left ( s+\dfrac{a_2}{a_1} \right )\sin \theta + \dfrac{\varphi' \cos^2\theta}{\kappa(s)a_1}.
	\]
	
	If we take $\varphi'(s)=0$, which means $\varphi(s)=\varphi_0 \in \mathbb{R}$, and we use the notations  
	\[
	\cos \varphi_0 = b_1, \; \sin \varphi_0 = b_2,
	\]
	where $b_1, b_2  \in [-1,1]$ and $b_1^2+b_2^2=1$, by computing the curvature  $\kappa(s)= \left\|\gamma''(s) \right\|= \left\| T'(s)\right\|$, we find
	\[
	\kappa(s)=\pm \theta'(s), \; \theta \in \left [ -\dfrac{\pi}{2}, \dfrac{\pi}{2} \right ].
	\]
	Now, using Theorem \ref{Th R-C}, we obtain 
	\[
	\theta(s)= \pm \arctan(a_1s+a_2) + C, 
	\]
	where $s \in [0,L], \; a_1, a_2, C \in  \mathbb{R}, \; a_1 \neq 0$. In conclusion, by direct computations, we find the formulas (\ref{C-T}).
	Consequently, by replacing these into the relations (\ref{G-C}) and (\ref{M-C}), we deduce that the curvatures associated with the ruled normal surface attached to the rectifying curve (\ref{Ex 1}) are described by:
	\begin{itemize}
		\item [1.] The \textit{Gaussian curvature}:
		\[
		K(s,u)=-\dfrac{4a_1^2(a_1s+a_2)^2[1+(a_1s+a_2)^2]^2}{\left\{[1+(a_1s+a_2)^2 \mp ua_1]^2+4u^2a_1^2(a_1s+a_2)^2 \right\}^2},
		\]
		\item [2.] The \textit{mean curvature}: 
		\[
		H(s,u)=ua_1^2\dfrac{(a_1s+a_2)^4-1 \pm \left\{ [2a_1-(a_1s+a_2)](a_1s+a_2)+1\right\}}{\left\{ [1+(a_1s+a_2)^2 \mp ua_1]^2 + 4u^2a_1^2(a_1s+a_2)^2\right\}^{3/2}},
		\]
	\end{itemize}
	where $s \in [0,L]$, $u \in \mathbb{R}$ $a_1, a_2 \in \mathbb{R}$ and $a_1 \neq 0$. \medskip

	The main results of the paper are the followings:
	\begin{theorem} \label{Rec not Bert}
		No curve can be simultaneously a rectifying curve and a Bertrand curve.
	\end{theorem}
	
	\begin{proof}
		From paper \cite{Izu-Take}, we know that the normal ruled surfaces attached to Bertrand curves are always minimal. According to Proposition \ref{Not minimal}, the normal ruled surfaces attached to rectifying curves cannot be minimal. We get a contradiction. 
	\end{proof}
	
	\begin{corollary}
		There are no curves having "a priori" the curvature and torsion given by
		\[
		\kappa(s)=\dfrac{1}{A+B(a_1s+a_2)}, \; \tau(s)=\dfrac{a_1s+a_2}{A+B(a_1s+a_2)},
		\]
		where $\forall \; s \in [0,L], A, B, a_1 \in \mathbb{R}^*$ and $a_2 \in \mathbb{R}$.
	\end{corollary}
	
	\begin{proof}
		A curve $\gamma(s)$ with above curvature and torsion are, at the same time, rectifying and Bertrand. According with Theorem \ref{Rec not Bert}, we conclude that there is no curve whose curvature and torsion can be expressed as shown above.  \medskip
	\end{proof}

	\noindent \textbf{Acknowledgements.} The author of this paper is indebted Professor Adela Mihai, whose very useful suggestions and comments helped to improve this research work.


	\noindent \textsc{Ana-Maria Boldeanu}\\[1mm]
	Transilvania University of Brasov\newline
	Department of Mathematics and Computer Science\newline
	Blvd. Iuliu Maniu, No. 50, 500091 Brasov, Romania.
	
	\noindent E-mail: ana.boldeanu@unitbv.ro
	
\end{document}